\documentclass[12pt,letterpaper]{article}
\usepackage{amsmath,amsfonts,amsthm,amssymb}

\swapnumbers 

\newtheorem{lemma}{Lemma}[section]
\newtheorem{prop}[lemma]{Proposition}
\newtheorem{corollary}[lemma]{Corollary}
\newtheorem{theorem}[lemma]{Theorem}

\theoremstyle{definition} 
\newtheorem{definition}[lemma]{Definition}

\newcommand\ea{{\bf EA}}
\newcommand\pog{{\bf POG}}

\begin{document}

\title{Tensor product of dimension effect algebras}

\author{ Anna Jen\v cov\'a  and Sylvia
Pulmannov\'{a}\footnote{ Mathematical Institute, Slovak Academy of
Sciences, \v Stef\'anikova 49, SK-814 73 Bratislava, Slovakia;
jenca@mat.savba.sk, pulmann@mat.savba.sk. Supported by
grant VEGA No.2/0069/16 and by the grant of the Slovak
Research and Development Agency grant No. APVV-16-0073.}}
\date{}

\maketitle

\begin{abstract}
Dimension effect algebras were introduced in (A. Jen\v cov\'a, S. Pulmannov\'a, Rep. Math. Phys. {\bf 62} (2008), 205-218), and it was proved that they are unit intervals in dimension groups. We prove that the effect algebra tensor product of dimension effect algebras is a dimension effect algebra, which is the unit interval in the unital abelian po-groups tensor product of the corresponding dimension groups.

\end{abstract}


\medskip


\section{Introduction}

In \cite{JePu}, the notion of a dimension effect algebra was introduced as a counterpart of the notion of a dimension group. Recall that a dimension group (or a Riesz group) is a directed, unperforated interpolation  group. By \cite{EHS}, dimension groups can be also characterized as direct limits of  directed systems of simplicial groups. In analogy with the latter characterization, dimension effect algebras were defined as direct limits of  directed systems of finite effect algebras with the Riesz decomposition property. It is well known that the latter class of effect algebras corresponds to the class of finite MV-algebras, and in analogy with simplicial groups, we call them simplicial effect algebras.
It turns out that dimension effect algebras are exactly the unit intervals in unital dimension groups, and simplicial effect algebras are exactly the unit intervals in unital simplicial groups. In \cite {JePu}, an intrinsic characterization of dimension effect algebras was found, and  also a categorical equivalence between countable dimension effect algebras and unital AF C*-algebras was shown \cite[Theorem 5.2]{JePu}.

In this paper we continue the study of dimension effect algebras. In particular, we study the tensor product of dimension effect algebras in the category of effect algebras.
We recall that the tensor product in the category of effect algebras exists, and its construction was described in \cite{Dvu}. We first prove that the tensor product of simplicial effect algebras  is again a simplicial effect algebra and is (up to isomorphism) the unit interval in the tensor product of the corresponding unital simplicial groups  (Theorem \ref{th:tenprodfmv}). Then we extend this result to any dimension effect algebras, using the fact that every dimension effect algebra is a direct limit of a directed system of simplicial effect algebras. Namely, we prove that the tensor product of dimension effect algebras  is a dimension effect algebra (Theorem \ref{th:tenproddimea}), and is (up to isomorphism) the unit interval in the tensor product of the corresponding dimension groups  (Corollary \ref{cor:unigroup}). We conjecture that this last statement holds more generally for tensor products of interval effect algebras and their universal groups.

We note that the categorical equivalence between effect algebras with RDP and interpolation groups proved in \cite[Theorem 3.8]{JePu}, or the known  constructions of tensor products in the category of interval effect algebras \cite{FGB} cannot be applied here, since the category of effect algebras is much larger than the category of effect algebras with RDP or interval effect algebras.

In the last section,  we apply our results to the interval ${\mathbb R}^+[0,1]$ and construct a directed system of simplicial groups that has this interval
 as its direct limit.

\section{Preliminaries}

The notion of an effect algebra was introduced by D.J. Foulis and M.K. Bennett in \cite{FoBe}. An alternative definition of so called \emph{D-poset} was introduced in \cite{KCh}.
Effect algebras and D-posets are categorically equivalent structures \cite{DvPu}.

\begin{definition}\label{de:ea} An \emph{effect algebra} is an algebraic system $(E;0,1,\oplus)$, where $\oplus$ is a partial binary operation and $0$ and $1$ are constants, such that the following axioms are satisfied for $a,b,c\in E$:
\begin{enumerate}
\item[{\rm(i)}] if $a\oplus b$ is defined the $b\oplus a$ is defined and $a\oplus b=b\oplus a$ (commutativity);
\item[{\rm(ii)}] if $a\oplus b$ and $(a\oplus b)\oplus c$ are defined, then $a\oplus(b\oplus c)$ is defined and $(a\oplus b)\oplus c=a\oplus(b\oplus c)$ (associativity);
\item[{\rm(iii)}] for every $a\in E$  there is a unique $a^{\perp}\in E$ such that $a\oplus a^{\perp}=1$;
\item[{\rm(iv)}] if $a\oplus 1$ is defined then $a=0$.
\end{enumerate}
\end{definition}

In what follows, if we write $a\oplus b$, $a,b\in E$, we tacitly assume that $a\oplus b$ is defined in $E$. The operation $\oplus$ can be extended to the $\oplus$-sum of finitely many elements by recurrence in an obvious way. Owing to commutativity and associativity, the element $a_1\oplus a_2\oplus \cdots \oplus a_n$ is ambiguously defined.   In any effect algebra a partial order can be defined as follows:   $a\leq b$ if there is $c\in E$ with $a\oplus c=b$. In this partial order, $0$ is the smallest and $1$ is the greatest element in $E$. Moreover, if $a\oplus c_1=a\oplus c_2$, then $c_1=c_2$, and we define $c=b\ominus a$ iff $a\oplus c=b$. In particular, $1\ominus a=a^{\perp}$ is called the \emph{orthosupplement} of $a$. We say that $a,b\in E$ are \emph{orthogonal}, written $a\perp b$, iff $a\oplus b$ exists in $E$. It can be shown that $a\perp b$ iff $a\leq b^{\perp}$.
An effect algebra which is a lattice with respect to the above ordering is called a \emph{lattice effect algebra}.

Let $E$ and $F$ be effect algebras. A mapping $\phi:E\to F$ is an \emph{effect algebra morphism} iff $\phi (1)=1$ and $\phi(e\oplus f)=\phi(e)\oplus \phi(f)$ whenever $e\oplus f$ is defined in $E$. 
The category of effect algebras with effect algebra morphisms will be denoted by $\ea$.

\subsection{Interval effect algebras and RDP}

Important examples of effect algebras are obtained in the following way. Let $(G,G^+,0)$ be a (additively written) partially ordered abelian group with a positive cone $G^+$ and neutral element $0$. For $a\in G^+$ define the interval $G[0,a]:=\{ x\in G: 0\leq x\leq a\}$.
Then $G[0,a]$ can be endowed with a structure of an effect algebra by defining $x\perp y$ iff $x+y\leq a$, and then putting $a\oplus b:=a+b$.  Effect algebras obtained in this way are called \emph{interval effect algebras}. We note that a prototype of effect algebras is the interval $[0,I]$ in the group of self-adjoint operators on a Hilbert space, so-called algebra of Hilbert space effects. Hilbert space effects play an important role in quantum measurement theory, and  the abstract definition was motivated by this example. On the other hand,  there are effect algebras that are not interval effect algebras, see e.g. \cite{Na}.

The partially ordered abelian group $G$ is \emph{directed} if  $G=G^+-G^+$.   An element $u\in G^+$  is an \emph{order unit} if for all $a\in G$, $a\leq nu$ for some $n\in {\mathbb N}$. If $G$ has an order unit $u$, it is directed, indeed, if $g\leq nu$, then $g=nu-(nu-g)$.  An element $u\in G^+$ is called a \emph{generating unit} if every $a\in G^+$ is a finite sum of (not necessarily different) elements of the interval $G[0,u]$. Clearly, a generating unit is an order unit, the converse may be false.

If $G$ and $H$ are partially ordered abelian groups, then a group homomorphism $\phi:G\to H$ is \emph{positive} if $\phi(G^+)\subseteq H^+$. An isomorphism $\phi : G\to H$ is an \emph{order isomorphism} if $\phi(G^+)=H^+$. If $G$ and $H$ have order units $u$ and $v$, respectively, then a positive homomorphism $\phi:G\to H$ is called \emph{unital} if $\phi(u)=v$. The category of partially ordered abelian groups having an order unit, with positive unital homomorphisms will be denoted by $\pog$.

Relations between interval effect algebras and partially ordered abelian groups are described in the following theorem, proved in \cite{BeFo}. Recall that a mapping $\phi:E\to K$, where $E$ is an effect algebra and $K$ is any abelian group, is called a \emph{$K$-valued measure} on $E$ if $\phi(a\oplus b)=\phi(a)+\phi(b)$ whenever $a\oplus b$ is defined in $E$.

\begin{theorem}\label{th:unigroup} Let $E$ be an interval effect algebra. Then there exists a unique (up to isomorphism) partially ordered directed abelian  group $(G,G^+)$ and an element $u\in G^+$ such that the following conditions are satisfied:
\begin{enumerate}
\item[{\rm(i)}] $E$ is isomorphic to the interval effect algebra $G^+[0,u]$.
\item[{\rm(ii)}] $u$ is a generating unit.
\item[{\rm(iii)}] Every $K$-valued measure $\phi:E\to K$ can be extended uniquely to a group homomorphism $\phi^*:G\to K$.
\end{enumerate}
\end{theorem}

The group $G$ in the preceding theorem is called a \emph{universal group} for $E$, and will be denoted by $G_E$.
In what follows we consider a property that ensures that a partially ordered group with order unit is the universal group for its unit interval. There are examples (see \cite[Example 11.3, 11.5]{FoGr}) that show that this is not true in general.

A partially ordered abelian group $G$ is said to have the \emph{Riesz interpolation property} (RIP), or to be an \emph{interpolation group}, if given $a_i,b_j$ ($1\leq i\leq m, 1\leq j\leq n$)
with $a_i\leq b_j$ for all $i,j$, there exists $c\in G$ such that $a_i\leq c\leq b_j$ for all $i,j$. The Riesz interpolation property is equivalent to the \emph{Riesz decomposition property} (RDP): given $a_i,b_j\in G^+$, ($1\leq i\leq m, 1\leq j\leq n$) with $\sum a_i=\sum b_j$, there exist $c_{ij}\in G^+$ with $a_i=\sum_jc_{ij}, b_j=\sum_ic_{ij}$. An equivalent definition of the RDP is as follows: given $a,b_i$ in $G^+$, $i\leq n$  with $a\leq \sum_{i\leq n}b_i$, there exist $a_i\in G^+$ with $a_i\leq b_i, i\leq n$, and  $a=\sum_{i\leq n}a_i$. To verify these properties, it is only necessary to consider the case $m=n=2$ (cf. \cite{Fuchs, Good}).

For interpolation groups we have the following theorem  \cite{Pu}, \cite[Theorem 3.5]{JePu}.

\begin{theorem}\label{th:rdpunigroup} Let $G$ be an interpolation group with order unit $u$. Put $E:=G^+[0,u]$. Then $(G,u)$ is the universal group for $E$.
\end{theorem}

In a similar way as for partially ordered abelian groups,  RDP can be defined for effect algebras. We say that an effect algebra $E$ has the \emph{Riesz decomposition property} (RDP) if one of the following equivalent properties is satisfied:
\begin{enumerate}
\item[(R1)] $a\leq b_1\oplus b_2\oplus \cdots \oplus b_n$ implies $a=a_1\oplus a_2\oplus \cdots \oplus a_n$ with $a_i\leq b_i, i\leq n$;
\item[(R2)] $\oplus_{i\leq m}a_i=\oplus_{j\leq m} b_j$, $m,n\in {\mathbb N}$, implies $a_i=\oplus_jc_{ij}, i\leq m$, and $b_j=\oplus_ic_{ij}, j\leq n$, where $c_{ij}\in E$.
\end{enumerate}
Similarly as for partially ordered groups, it suffices to consider the case $m=n=2$.

Let us remark that RIP can be also defined for effect algebras.
In contrast with the case of  partially ordered abelian groups, RIP and RDP  are not equivalent for effect algebras: RDP implies RIP, but there are examples of effect algebras with RIP which do not have RDP (e.g., the "diamond" is lattice ordered effect algebra that does not satisfy RDP, \cite{DvPu}).

It was proved by Ravindran \cite{Rav}, that every effect algebra with RDP is an interval effect algebra, and its universal group is an interpolation group.
Ravindran's result can be extended to a categorical equivalence between the category of effect algebras with RDP with effect algebra morphisms and the category  of interpolation groups with order unit with positive unital group homomorphisms, \cite[Theorem 3.8]{JePu}.

\section{Dimension groups and dimension effect algebras}

In this section, we study dimension groups and their effect algebra counterpart, introduced in  \cite{JePu}. These are interpolation groups with some additional properties.

 A partially ordered abelian group $G$ is called \emph{unperforated} if given $n\in {\mathbb N}$ and $a\in G$, then $na\in G^+$ implies $a\in G^+$. Every Archimedean, directed abelian group is unperforated \cite[Proposition 1.24]{Good}, and also every lattice ordered abelian group is unperforated \cite[Proposition 1.22]{Good}.

\begin{definition}\label{de:dingr} {\rm \cite{Good}} A partially ordered group $G$ is a \emph{dimension group} (or a \emph{Riesz group}) if it is directed, unperforated and has the interpolation property.
\end{definition}

A simple example of a dimension group is as follows.

\begin{definition}\label{de:simplicgr} {\rm \cite[Definition p. 183]{Good}, \cite{GoHa}} A \emph{simplicial group} is any partially ordered abelian group that is isomorphic (as partially ordered abelian group) to ${\mathbb Z}^n$ (with the product ordering) for some nonnegative integer $n$. A \emph{simplicial basis} for a simplicial group
$G$ is any basis $(x_1,\ldots,x_n)$ for $G$ as a free abelian group such that $G^+={\mathbb Z}^+x_1+\cdots +{\mathbb Z}^+x_n$.
\end{definition}

It was proved by Effros, Handelman and Shen \cite{EHS} that the dimension groups with order unit  are precisely the direct limits of directed systems of simplicial groups with an order unit  in the category $\pog$.

Note that an element $v\in {\mathbb Z}^r$ is an order unit if and only if all of its coordinates are strictly positive. In this case, the interval $({\mathbb Z}^+)^r[0,v]$ is the direct product of finite chains $(0,1,\ldots,v_i), i=1,2,\ldots,r$ and therefore is a finite effect algebra with RDP. Conversely, every finite effect algebra with RDP is a unit interval in a simplicial group. Below, such effect algebras will be called \emph{simplicial}.

 In analogy with dimension groups, in \cite{JePu}, direct limits of directed systems of simplicial effect algebras have been called \emph{dimension effect algebras}. It was shown  that an effect algebra is a dimension effect algebra if and only if its universal group is a dimension group. An intrinsic characterization of dimension effect algebras was found in \cite[Theorem 4.2]{JePu}.

For the convenience of the readers, we give a short description of the directed system and direct limit of effect algebras \cite[Definition 1.9.36]{DvPu}.

A \emph{directed system of effect algebras} is a family $A_I:=(A_i; (f_{ij}: A_j\to A_i, i,j\in I, j\leq i)$ where $(I,\leq)$ is a directed set, $A_i$ is an effect algebra for each $i\in I$, and $f_{ij}$ is a morphism  such that
\begin{enumerate}
\item[(i1)] $f_{ii}=id_{A_i}$ fir every $i\in I$;
\item[(i2)] if $m\leq j\leq i$ in $I$, then $f_{ij}f_{jm}=f_{im}$.
\end{enumerate}

Let $A_I$ be a directed system of effect algebras, then $\underline{f}:=(A; (f_i:A_i\to A; i\in I))$ is called the \emph{direct limit} of $A_I$ iff the following conditions hold:
\begin{enumerate}
\item[(ii1)] $A$ is an effect algebra; $f_i$ is a morphism for each $i\in I$;
\item[(ii2)] if $j\leq i$ in $I$, then $f_if_{ij}=f_j$ (i.e., $\underline{f}$ is compatible with $A_I$);
\item[(ii3)] if $\underline{g}:=(B; (g_i:A_i\to B, i\in I))$ is any system compatible with $A_I$, then there exists exactly one morphism $g:A\to B$ such that $gf_i=g_i$, for every $i\in I$.
\end{enumerate}

It was proved (cf. \cite[Theorem  1.9.27]{DvPu}) that the direct limit in the category of effect algebras exists. A sketch of the construction of the direct limit is as follows.
  Let $A= \dot{\cup}_{i\in I}A_i$ be the disjoint union of $A_i, i\in I$. Define a relation $\equiv$ on $A$ as follows. Put $a\equiv b$ ($a\in A_i, b\in A_j$) if there exists a $k\in I$ with $i,j\leq k$ such that $f_{ki}(a)=f_{kj}(b)$ in $A_k$. Then $\equiv$ is an equivalence relation, and the quotient $\bar{A}:=A/\equiv$ can be organized into an effect algebra
with the operation $\oplus$ defined as follows:  let $\bar{a}$ denotes the equivalence class corresponding to $a$. For $a\in A_I, b\in A_j$, $\bar{a}\oplus \bar{b}$ is defined iff there is $k\in I$, $i,j\leq k$ such that $f_{ki}(a)\oplus f_{kj}(b)$ exists in $A_k$, and then $\bar{a}\oplus \bar{b}=\overline{(f_{ki}(a)\oplus f_{kj}(b))}$ in $\bar{A}$. For every $i\in I$, define $f_i: A_i\to A/\equiv$ as the natural projection $f_i(a)=\bar{a}$.
 Then $\lim_{\rightarrow} A:=(\bar{A}; f_i:A_i\to \bar{A}, i\in I)$ is the desired direct limit.\label{pg:direct}

 From this construction, it can be derived that properties involving finite number of elements, such as RDP or being a dimension effect algebra (cf. the characterization
in \cite[Thm. 4.2]{JePu}), are preserved under direct limits in $\ea$.

\section{Tensor product of dimension effect algebras}

The tensor product in the category ${\bf EA}$ is defined below as an universal bimorphism.
We will show that such a tensor product always exists and that it is essentially given by the construction in  \cite{Dvu}, see also \cite[Chap. 4.2]{DvPu}.

Let $E,F,L$ be effect algebras. A mapping $\beta: E\times F\to L$ is called a \emph{bimorphism} if
\begin{enumerate}
\item[(i)] $a,b\in E$ with $a\perp b$, $q\in F$ imply $\beta(a,q)\perp \beta(b,q)$ and $\beta(a\oplus b,q)=\beta(a,q)\oplus \beta(b,q)$;
\item[(ii)] $c,d\in F$ with $c\perp d$, $p\in E$ imply $\beta(p,c)\perp \beta(p,d)$ and $\beta(p,(c\oplus d))=\beta(p,c)\oplus \beta(p,d)$;
\item[(iii)] $\beta(1,1)=1$.
\end{enumerate}

\begin{definition}\label{de:tenprodea} Let $E$ and $F$ be effect algebras. A pair $(T,\tau)$ consisting of an effect algebra $T$ and a bimorphism $\tau:E\times F \to T$ is said to be the \emph{tensor product} of $E$ and $F$ if whenever
$L$ is an effect algebra and $\beta : E\times F \to L$ is a bimorphism, there exists a unique morphism $\phi :T\to L$ such that $\beta=\phi \circ \tau$.

\end{definition}

It is clear that if the tensor product exists, it is unique up to isomorphism. We will use the notation $E\otimes F$ for the effect algebra $T$ and $\otimes$ for the bimorphism $\tau$: $\tau(e,f)=e\otimes f\in E\otimes F$.

\begin{theorem}\label{th:tenprodea}
The tensor product always exists in ${\bf EA}$.
\end{theorem}

\begin{proof} The  theorem was essentially proved in \cite[Theorem 7.2]{Dvu}, see also \cite[Theorem 4.2.2]{DvPu}. There a somewhat different definition of
a tensor product is considered and the bimorphisms are assumed nontrivial, that is, the target algebra is required to satisfy $0\ne 1$. If at least one such bimorphism exists, it is easy to see that \cite{Dvu} provides a construction of a tensor product in our sense. On the other hand, if there are no nontrivial bimorphisms, then the
tensor product is given by the one-element effect algebra $\{0=1\}$ and the unique bimorphism $E\times F\to \{0\}$.

\end{proof}

The tensor product of dimension groups in the category $\pog$ was studied by Goodearl and Handelman \cite{GoHa} and it was proved that such a tensor product is a dimension group as well. Recall that the tensor product of $G_1$ and $G_2$ in $\pog$ can be constructed as the
abelian group tensor product $G_1\otimes G_2$, endowed with the positive cone $G_1^+\otimes  G_2^+$, generated by simple tensors of positive elements.

Our aim in this section is to describe the tensor product of dimension effect algebras in the category ${\bf EA}$. Note that we cannot directly apply the above result
via the categorical equivalence of \cite[Theorem 3.8]{JePu}, since the category ${\bf EA}$ is much larger than the category of effect algebras with RDP.

 We first consider the case of simplicial effect algebras.
Let $E$ and $F$ be simplicial effect algebras, with atoms
\[
(e_1,\dots,e_n),\qquad  (f_1,\dots,f_m)
\]
and unit elements
\[
u=\sum_iu_ie_i, \qquad v=\sum_jv_jf_j,
\]
respectively. Then $G_E$ and $G_F$ are simplicial groups and $G_E\otimes G_F$ is a simplicial group with generators
\[
g_{ij}=e_i\otimes f_j, i=1,\dots,n; j=1,\dots,m.
\]
Hence the unit interval $G_E\otimes G_F[0,u\otimes v]$ is a simplicial effect algebra with atoms $g_{ij}$ and
and unit element $w=\sum_{i,j} u_iv_jg_{ij}$.

\begin{theorem}\label{th:tenprodfmv} Tensor product of  simplicial effect algebras in the category ${\bf EA}$  is a simplicial effect algebra, namely
\[
E\otimes F\simeq  G_E\otimes G_F[0,u\otimes v].
\]
\end{theorem}

\begin{proof} Let $G$ denote the simplicial effect algebra on the right hand side.
Obviously, (bi)morphisms on simplicial effect algebras are uniquely determined by their values on the atoms.
Let $\tau: E\times F\to G$ be the bimorphism determined by
\[
\tau(e_i,f_j)=g_{ij}, \qquad i=1,\dots,n, \ j=1,\dots,m.
\]
We need to prove that for any effect algebra $H$ and bimorphism $\beta: E\times F\to H$, there is a morphism $\psi: G\to H$, such that
\[
\psi(g_{ij})=\beta(e_i,f_j),\qquad i=1,\dots,n, \ j=1,\dots,m.
\]
Since $g_{ij}$ generate $G$, uniqueness of such a morphism is clear.
So let $z\in G$, then $z=\sum_{i,j}z_{ij}g_{ij}$, for $z_{ij}\le u_iv_j$ for all $i$ and $j$. There are nonnegative integers $q_{ij}$, $r_{ij}$ such that
\[
z_{ij}= v_jq_{ij}+r_{ij},\qquad r_{ij}< v_j,
\]
then since $v_j q_{ij}\le z_{ij}\le u_iv_j$, we have $q_{ij}\le u_i$, with equality only if $r_{ij}=0$.
Then  $a_j:=\sum_iq_{ij} e_i\in E$ and $r_{ij}f_j\in F$.
We have
\begin{align*}
z&=\sum_j (\sum_iq_{ij}v_j g_{ij}+ \sum_{i} r_{ij}g_{ij})\\
&= \sum_j \tau(a_j,v_j f_j)+ \sum_{i, r_{ij}>0}\tau(e_i,r_{ij}f_j)
\end{align*}
 Put $a'_j:=\sum_{i, r_{ij}>0} e_i$, then $a_j\perp a_j'$. Now we can write
\begin{align*}
H\ni 1=\beta(u,v)&=\sum_j\beta(u,v_jf_j)\\
&= \sum_j \left[\beta(a_j, v_jf_j)+ \beta(a'_j, v_jf_j)+\beta(u-(a_j+a_j'), v_jf_j)\right]\\
&= \sum_j [\beta(a_j, v_jf_j)+\sum_{i}\beta (e_i, r_{ij}f_j)+ \sum_{i, r_{ij}>0}\beta(e_i, (v_j-r_{ij})f_j)\\
&+\beta(u-(a_j+a_j'), v_jf_j)]
\end{align*}
It follows that
\begin{align*}
\sum_{i,j} z_{ij}\beta(e_i,f_j)&=\sum_{i,j} [q_{ij}v_j\beta(e_i,f_j)+r_{ij}\beta(e_i,f_j)]\\
&=\sum_j [\beta(a_j, v_jf_j)+\sum_i\beta (e_i, r_{ij}f_j)]
\end{align*}
is a well defined element in $H$ and we may put
\[
\psi(z)=\sum_{i,j} z_{ij}\beta(e_i,f_j),
\]
which clearly defines a morphism $G\to H$.

\end{proof}

Let
\begin{align*}
A_I&=(A_i; (f_{ij}:A_j \to A_i); i,j\in I, j\leq i),\\
B_J&=(B_k; (g_{k\ell}:B_\ell \to B_{k}); k,\ell \in J, \ell\leq k)
\end{align*}
be directed systems of simplicial effect algebras. Let us define the index set $(\mathcal I,\leq)$ as the product $I\times J$ with pointwise ordering.
By the previous theorem, each $A_i\otimes B_k$, $(i,k)\in \mathcal I$ is a simplicial effect algebra. Let $(j,\ell)\in \mathcal I$ be such that $(j,\ell)\leq (i,k)$, then
 we have morphisms $f_{ij}:A_j\to A_i$ and $g_{k\ell}: B_\ell\to B_k$. For $a\in A_j$, $b\in B_\ell$, put $\beta(a,b)=f_{ij}(a)\otimes g_{k\ell}(b)\in A_i\otimes B_k$, this  defines a bimorphism $A_j\times B_\ell\to A_i\otimes B_k$.  By properties of tensor product, this extends to a unique morphism
$f_{ij}\otimes g_{k\ell}: A_j\otimes B_{\ell}\to A_i\otimes B_k$.

\begin{theorem}\label{th:directed} Let
\begin{eqnarray*}
& A_I\otimes B_J:=(A_i\otimes B_k; (f_{ij}\otimes g_{k\ell}:A_j\otimes B_\ell \to A_i\otimes B_{k}), \\
& (i,k), (j,\ell)\in \mathcal I, (j,\ell)\leq (i,k)).
\end{eqnarray*}
Then $A_I\otimes B_J$ is a directed system of simplicial effect algebras.
\end{theorem}

\begin{proof} We have to check properties (i1) and (i2). For (i1), note that  $f_{ii}=id_{A_i}$, $g_{kk}=id_{B_k}$ imply $f_{ii}\otimes g_{kk}=id_{A_i\otimes B_k}$.
For (i2), let  $(m,n)\leq (j,\ell)\leq (i,k)$. Then
\[
m\leq j\leq i \ \implies \ f_{ij}f_{jm}=f_{im}
\]
\[n\leq \ell \leq k\ \implies \ g_{k\ell} g_{\ell n}=g_{kn}
\]
and for $a_m\in A_m, b_n\in B_n$,
\begin{eqnarray*}
(f_{ij}\otimes g_{k\ell})(f_{jm}\otimes g_{\ell n})(a_m\otimes b_n) &=& (f_{ij}\otimes g_{k\ell})(f_{jm}(a_m)\otimes g_{\ell n}(b_n))\\
&=& f_{ij}f_{jm}(a_m)\otimes g_{k\ell}g_{\ell n}(b_n)\\
&=& f_{im}(a_m)\otimes g_{kn}(b_n)\\
&=& f_{im}\otimes g_{kn}(a_m\otimes b_n).
\end{eqnarray*}
Since this holds on simple tensors, it extends to whole $A_m\otimes B_n$.

\end{proof}

\begin{theorem}\label{th:tenproddimea} Let $A_I, B_J$ be directed systems of simplicial effect algebras, and let $(\bar{A};(f_i:A_i\to \bar{A}, i\in I))$ and $(\bar{B}; (g_j:B_j\to \bar{B}, j\in J))$ be their corresponding direct limits. Then $(\bar{A}\otimes \bar{B}; (f_i\otimes g_j:A_i\otimes B_j \to \bar{A}\otimes \bar{B}, i\in I, j\in J))$
is the direct limit of $A_I\otimes B_J$.
\end{theorem}

\begin{proof} We have to check properties (ii1), (ii2) and (ii3). The first one is clear: since $\bar{A}$, $\bar{B}$ are effect algebras, $\bar{A}\otimes \bar{B}$ is an effect algebra as well. To prove compatibility, let $(j,\ell) \leq(i,k)$.  Then  $j\leq i, \ell\leq k$ and we have
\[
(f_i\otimes g_k)(f_{ij}\otimes g_{k\ell})= f_if_{ij}\otimes g_kg_{k\ell}=f_j\otimes g_{\ell}.
\]
For (ii3), let $(C; (h_{ij}:A_i\otimes B_j \to C, i\in I, j\in J))$ be another system compatible with $A_I\otimes B_J$ (i.e., $h_{ik}(f_{ij}\otimes g_{k\ell})=h_{j\ell}, j\leq i, \ell\leq k$). Let $a\in \bar A$, $b\in \bar B$. Since $\bar A$ and $\bar B$ are direct limits, there are some indices $i\in I$, $k\in J$ and elements $a_i\in A_i$ and $b_k\in B_k$ such that
$a=f_i(a_i)$ and $b=g_k(b_k)$, see the construction on page \pageref{pg:direct}.
Define $h(a,b):=h_{ik}(a_i\otimes b_k)$. Then $h:\bar{A}\times \bar{B} \to C$  is a bimorphism, which extends to a morphism $\bar{h}:\bar{A}\otimes \bar{B} \to C$.
\end{proof}

\begin{corollary}\label{cor:unigroup}
  Let $E$ and $F$ be dimension effect algebras, and let $G_E$ and $G_F$ be their universal groups with units $u_E$ and $u_F$. Then  the tensor product  $E\otimes F$ is isomorphic to the unit interval $[0,u_E\otimes u_F]$ in the tensor product $G_E\otimes G_F$ of their universal groups, that is
  \[
  G_E[0,u_E]\otimes G_F[0,v_F]\simeq G_E\otimes G_F[0,u_E\otimes u_F].
  \]

\end{corollary}

\begin{proof} Let $E=\bar A$, $F=\bar B$ be direct limits of directed systems  $A_I$ and $B_J$. Each $A_i$, $i\in I$ and $B_k$, $k\in J$ is a simplicial effect algebra and
$G_{A_i}$, $G_{B_k}$ are simplicial groups. By  \cite[Theorem 4.1]{JePu}, we obtain that
$G_E$ is a direct limit of $(G_{A_i}, f_{ij}^*)$, where $f_{ij}^*$ are the unique morphisms in $\pog$, extending $f_{ij}$, similarly for $G_F$.

By Theorem \ref{th:tenprodfmv}, $A_i\otimes B_k$ is a simplicial effect algebra and  $G_{A_i\otimes B_k}\simeq G_{A_i}\otimes G_{B_k}$.
By Theorem \ref{th:tenproddimea}, $E\otimes F$ is the direct limit of the directed system
$A_I\otimes B_J$. Since $A_I\otimes B_J$ has RDP, it follows by \cite[Theorem 4.1]{JePu}
 that the universal group $G_{E\otimes F}$ is a direct limit of the system of universal groups
 \[
 \{ G_{A_i\otimes B_k}\simeq G_{A_i}\otimes  G_{B_k}, (f_{ij}\otimes g_{k\ell})^*\simeq f_{ij}^*\otimes g_{k\ell}^*\},
 \]
  Using \cite[Lemma 2.2]{GoHa}, we obtain
 \[
 G_{E\otimes F}\simeq G_E\otimes G_F,\qquad u_{E\otimes F}=u_E\otimes u_F.
 \]
\end{proof}

\section{Conclusions and a conjecture}

We have proved that the $\ea$ tensor product of dimension effect algebras is  again a dimension effect algebra. The tensor product $E\otimes F$ is proved to be the
 direct limit of a directed system  of simplicial effect algebras, obtained as a ''tensor product'' of the directed systems corresponding to dimension effect algebras $E$ and $F$.

It is also proved that $E\otimes F$ is (isomorphic to) the unit interval in the $\pog$ tensor product of the corresponding universal groups $G_E$ and $G_F$. We conjecture  that this
is true for general interval effect algebras. Note that in the category of \emph{interval} effect algebras, the tensor product  exists \cite[Theorem 9.1]{FGB} and our conjecture says that it is (isomorphic to) the $\ea$ tensor product.

A special class of interval effect algebras are the algebras with RDP. It is again an open question whether in this case the $\ea$ tensor product has RDP.  If our conjecture is true, $E\otimes F$ is the unit interval in the $\pog$ tensor product of groups with RDP. As it was shown in \cite[cf. Remark 2.13]{W}, the $\pog$ tensor product of groups with RDP might not have RDP, but in the presence of generating units, RDP holds in an asymptotic form in the sense of \cite{Par}.

\section{An example: $\mathbb R[0,1]$}

Let us consider the interval $[0,1]$ in  $(\mathbb R,\mathbb R^+,0)$. This is clearly a dimension group with order unit $1$
 and hence the interval $[0,1]$ is a dimension effect algebra.
 It was proved in \cite{Pu2} that the ${\bf EA}$ tensor product $[0,1]\otimes [0,1]$ is not lattice ordered and thus not isomorphic to $[0,1]$. By our results, $[0,1]\otimes [0,1]$ is a dimension effect algebra, which is the interval $R\otimes R[0,1\otimes 1]$.
  Note that the fact that the $\pog$ tensor product $R\otimes R$ is not lattice ordered was shown in \cite{W}.

  As an example, we will present $[0,1]$ as a direct limit of a directed system of simplicial effect algebras. The tensor product $[0,1]\otimes [0,1]$ is then obtained as a direct limit as in Theorem \ref{th:tenproddimea}.

We first need to introduce some notations. For any $n$-tuple
\[
A=(x_1,\dots,x_n)
\]
  of elements in $\mathbb R^+$, let $f_A$ denote the positive  group homomorphism
\[
f_A: \mathbb Z^n\to \mathbb R,\quad  e^n_i\mapsto x_i,\ i=1,\dots,n
\]
and let
\[
L(A):= f_A(\mathbb Z^n), \quad L(A)^+:=  f_A((\mathbb Z^n)^+),\quad L_>(A)^+:=f_A(\mathbb (Z_>^n)^+),
\]
where $\mathbb (Z_>^n)^+:= \{\sum_i z_i e^n_i$ with $z_i>0$ for all $i=1,\dots,n\}$.
We also use the notations
\[
Q(A):=Lin_{\mathbb Q}(A),\quad Q(A)^+:=Q(A)\cap\mathbb R^+.
\]

Let us  define the index set as
\[
\mathcal I:=\{A\subset [0,1], \mbox{finite, }\mathbb Q-\mbox{linearly independent, } 1\in L_>(A)^+\}.
\]
Any $A\subset \mathbb R^+$ with cardinality $n$  can be identified with the $n$-tuple of its elements $(x_1,\dots,x_n)$, indexed so that $x_1<\dots <x_n$.
For $A,B\in \mathcal I$, write $B\preceq A$ if $B\subset L(A)^+$. It is easy to see that $\preceq$ is a preorder in $\mathcal I$.

 \begin{prop}\label{prop:directed} $(\mathcal I,\preceq)$ is directed.
 \end{prop}

For the proof, we need some lemmas.

\begin{lemma}\label{lemma:sums}
Let $B=(y_1,\dots,y_k)$ be a tuple of elements in $\mathbb R^+$.  Assume that for some $1\le N<  k$,
\[
\sum_{i=1}^Ny_i=\sum_{i=N+1}^ky_i.
\]
 Then there is some tuple $A=(x_1,\dots, x_l)$ of elements in $Q(B)^+$ such that $l<k$ and $y_i\in L(A)^+$, $i=1,\dots,k$.

\end{lemma}

\begin{proof}
We  proceed by induction on $k$.
By the assumptions, we see that $k$ is at least 2, in which case we have $y_1=y_2$.
Put  $A:=\{y_2\}$ and we are done.

Now let $k>2$ and assume that the assertion is true for tuples of length $k-1$. By reindexing and rearranging the sums, we may assume that $y_k=\min\{y_1,\dots,y_k\}$. Put $y_1':=y_1-y_k$, then $y_1'\in Q(B)^+$ and we have the equality
\[
y_1'+y_2+\dots +y_N=y_{N+1}+\dots+ y_{k-1}
\]
containing only $k-1$ elements. By the induction hypothesis, there is some tuple  $A'=(x_1,\dots,x_{l'})$ with elements in
$Q(B)^+$ and  $l'< k-1$, and
some
 $(k-1)\times l'$ matrix $Z'$ with values in nonnegative integers such that
 \[
y_1'=f_{A'}(z'_{1\cdot}),\quad y_i=f_{A'}(z'_{i\cdot}),\ i=2,\dots,k-1,
 \]
here $z'_{i\cdot}$ denotes the $i$-th row of $Z'$.  Let $A=(x_1,\dots,x_{l'},y_k)$ and
\[
Z=\left(\begin{array}{cc}
 Z' & \begin{array}{c} 1\\ 0\\ \vdots\\ 0\end{array}\\
 0 & 1
\end{array}\right).
\]
Then $A$ is an $l$-tuple of elements in $Q(B)^+$, $l=l'+1<k$ and $y_i=f_A(z_{i\cdot})\in L(A)^+$ for all $i$.

\end{proof}

\begin{lemma}\label{lemma:basis_positive}
Let $B=(y_1,\dots,y_k)$ be a tuple of elements in $\mathbb R^+$. Then there is a $\mathbb Q$-linearly independent tuple  $A=(x_1,\dots,x_n)$ of elements in $Q(B)^+$
 such that $y_i\in  L(A)^+$, $i=1,\dots,k$.

\end{lemma}

\begin{proof} If $B$ is $\mathbb Q$-linearly independent, there is nothing to do.  Otherwise, there are some $r_i\in\mathbb Q$ such that
$\sum_i r_i y_i=0$ with some $r_i\ne 0$. Clearly, by multiplying by a common denominator, we may assume that $r_i\in \mathbb Z$.
Assume that the elements are arranged in such a way that
\[
r_i\left\{\begin{array}{cc} >0 & \mbox{ for } i=1,\dots,N\\
  <0 & \mbox{ for } i=N+1,\dots M\\
 =0 & \mbox{ for } i=M+1,\dots k.
\end{array}\right.
\]
Put $p_i=\Pi_{i\ne j\le M} |r_j|$ and let $y'_i=\frac{y_i}{p_i}$ for $i=1,\dots,M$.
Clearly, $y_1',\dots, y_M'\in Q(B)^+$. Then by multiplying the equality by $\Pi_{j=1}^M|r_i|^{-1}$, we obtain
\[
\sum_{i=1}^Ny'_i=\sum_{i=N+1}^My'_i.
\]
 Applying Lemma \ref{lemma:sums}, there is some $l$-tuple  $A'=(x_1',\dots, x_l')\in Q(B)^+$ with $l<M$ such that
 $y_i'\in L(A')^+$ for $i=1,\dots,M$, so that also $y_i=p_iy_i'\in L(A')^+$, $i=1,\dots,M$.

 We now repeat the same process with $B'=(x_1',\dots,x_l',y_{M+1},\dots,y_k)$. Since $Q(B')=Q(B)$ and $|B'|<k$, after a finite number of steps we obtain a
$\mathbb Q$-linearly independent set $A=\{x_1,\dots,x_n\}$ with the required properties.

\end{proof}

\noindent
\textit{Proof of Proposition \ref{prop:directed}}.
Let $B,C\in \mathcal I$, then by Lemma \ref{lemma:basis_positive} there is some $\mathbb Q$-linearly independent
tuple $A=(x_1<\dots<x_n)$ of elements in  $Q(B\cup C)^+$ such that $B\cup C\subset L(A)^+$.
 By assumptions, $1\in L_>(B)^+\subset L(A)^+$, so that
 $1=\sum_i z_ix_i$ for
  unique coefficients $z_1,\dots, z_n\in \mathbb Z^+$. Assume that $z_{i_0}=0$ for some $i_0$.
  Let $B=(y_1<\dots<y_k)$. There are some positive integers $v_1,\dots,v_k$ such that $1=\sum_{j=1}^k v_jy_j$ and some
   nonnegative integers $w^j_1,\dots,w^j_n$ such that $y_j=\sum_iw^j_ix_i$. It follows that
   \[
1=\sum_{j=1}^k v_jy_j=\sum_i(\sum_j v_j w^j_i )x_i=\sum_i z_ix_i,
   \]
so that $\sum_j v_j w^j_{i}=z_i$, in particular, $\sum_j v_j w^j_{i_0}=0$. Since all $v_j$ are positive, this implies that
 $w^j_{i_0}=0$ for all $j$ and we have
 \[
y_j=\sum_{i\ne i_0} w^j_ix_i.
 \]
Hence $B\subset L(A\setminus \{x_{i_0}\})^+$, similarly also $C\subset L(A\setminus \{x_{i_0}\})^+$. It follows that we may assume that
$1\in L_>(A)^+$. This means that $1=\sum_i z_i x_i$ for positive integers $z_i$, which implies that we must have $0<x_i\le 1$. It follows that
 $A\in \mathcal I$ and $\mathcal I$ is directed.

\qed

We now construct a directed system of simplicial effect algebras.
Let $A\in \mathcal I$. Since $A$ is $\mathbb Q$-linearly independent, $f_A$ is a $\pog$  isomorphism onto its range.  Let
$E_A$ be  the interval  $[0,f_A^{-1}(1)]$ in $\mathbb Z^{|A|}$ and let
$g_A=f_A|_{E_A}$. Then $g_A$ is an effect algebra isomorphism onto the interval $[0,1]$ in $(L(A),L(A)^+,0)$. Let $B\in \mathcal I$,
 $B\preceq A$, then since $L(B)^+\subseteq L(A)^+$, we have $g_B(E_B)\subseteq g_A(E_A)$. Put
 \[
g_{AB}: E_B\to E_A, \qquad g_{AB}=g_A^{-1}g_B,
 \]
then it is clear that
\[
\mathcal E=(E_A, A\in \mathcal I; g_{AB}, B\preceq A)
\]
is a directed system of  simplicial effect algebras.

 \begin{prop} $([0,1]; g_A, A\in \mathcal I)$ is  the direct limit of $\mathcal E$.

 \end{prop}

\begin{proof}  It is clear that $([0,1]; g_A, A\in \mathcal I)$  is compatible with $\mathcal E$.
Note also that any $x\in [0,1]$ is contained in the range of some $g_A$. Indeed, assume that $x\in \mathbb Q\cap[0,1]$, then
$x=\tfrac mn$ for $n\in \mathbb N$, $m\in \mathbb Z^+$. Let $A=\{\tfrac1n\}$, then $A\in \mathcal I$ and we have $E_A=[0,n]_{\mathbb Z}$, $x=g_A(m)$. If $x\notin \mathbb Q$, then $A=\{x,1-x\}\in \mathcal I$ and $x\in A\subset g_A(E_A)$.

Now let $E$ be an effect algebra and let $k_A: E_A\to E$ be a morphisms for $A\in \mathcal I$, such that
$(E; k_A, A\in \mathcal I)$  is compatible with $\mathcal E$.
Let   $x\in [0,1]$ be in  the range
 of $g_A$ and put
\[
\psi(x)=k_A(g_A^{-1}(x)).
\]
Assume that  $B\in \mathcal I$ is  such that $x$ is also in  the range of $g_B$ and let $C\in \mathcal I$ be such that
 $A,B\preceq C$. Then $g_A(E_A)\subseteq g_C(E_C)$ and by compatibility
 \[
k_A(g_A^{-1}(x))=k_Cg_{CA}(g_A^{-1}(x))=k_C(g_C^{-1}(x)).
 \]
Similarly we obtain that $k_B(g_B^{-1}(x))=k_C(g_C^{-1}(x))$, hence $\psi$ is a well defined map.

Let $I=\{1\}$, then clearly $I\in \mathcal I$, $E_I=\{0,1\}\subset \mathbb Z$ and we have
\[
\psi(0)=k_I(0)=0,\qquad \psi(1)=k_I(1)=1,
\]
since  $k_I$ is an effect algebra morphism.
Further, let $x_1,x_2,x\in [0,1]$ be such that $x=x_1+x_2$. Let $A\in \mathcal I$ be such that $x_1,x_2\in g_A(E_A)$, then clearly also
 $x\in  g_A(E_A)$ and we have $g_A^{-1}(x_1)+g_A^{-1}(x_2)=g_A^{-1}(x)$, since $g_A$ is an isomorphism onto its range.
Hence
\[
\psi(x)=k_A(g_A^{-1}(x))=k_A(g_A^{-1}(x_1)+g_A^{-1}(x_2))=\psi(x_1)+\psi(x_2).
\]
This proves that $\psi$ is an effect algebra  morphism  $[0,1]\to E$. Further, for any $A\in \mathcal I$ and  $z\in E_A$,
\[
k_A(z)=k_A(g_A^{-1}g_A(z))=\psi g_A(z),
\]
so that $k_A=\psi g_A$. Since $\psi$ is obviously the unique map $[0,1]\to E$ with this property, this proves the statement.

\end{proof}

\end{document}